\documentclass[a4paper,12pt]{amsart}

\usepackage{amsmath,amsfonts,amsthm,url}
\usepackage{graphicx}
\usepackage{lineno}
\usepackage{slashbox}
\usepackage[table]{xcolor}
\newtheorem{theorem}{Theorem}[section]

\newtheorem{proposition}[theorem]{Proposition}

\newtheorem*{conjecture*}{Conjecture}

\theoremstyle{definition}

\theoremstyle{remark}
\newtheorem{remark}[theorem]{Remark}

\newcommand{\chooseversion}[2]
{#1}

\newcommand{\NN}{\mathbb{N}}
\newcommand{\ZZ}{\mathbb{Z}}

\newcommand{\K}[3]{g_{#1, #2, #3}}                

\newcommand{\SSK}[3]{\overline{\overline{g}}_{#1, #2, #3}}   

\newcommand{\ep}{\emptyset} 			
\newcommand{\cut}[1]{\overline{#1}}             
\newcommand{\cutt}[1]{\widehat{#1}}             

\newcommand{\SetL}{\mathcal{L}}                 
\newcommand{\Dom}{\mathcal{D}}                  
\newcommand{\semigroup}{\mathcal{S}}            
\newcommand{\polcol}{Q^-}     			
\newcommand{\polrow}{Q^+}     			
\newcommand{\op}[2]{\oplus (#1|#2)}  		
\newcommand{\scalar}[2]{\left\langle  #1 \, \middle| \, #2  \right\rangle} 

\author[Briand]{Emmanuel Briand}
\address{Departamento de Matem\'atica Aplicada I, Escuela T\'ecnica Superior de Ingenier\'ia Inform\'atica,
Avda.\ Reina Mercedes, S/N,
41012 Sevilla, Espa\~na}
\thanks{E. Briand and M.Rosas are partially supported by projects MTM2013--40455--P, FQM--333, P12--FQM--2696 and FEDER} 
\author[A. Rattan]{Amarpreet Rattan}
\address{Department of Economics, Mathematics and Statistics, Birkbeck,
University of London, London, UK, WC1E 7HX}
\author[M. Rosas]{Mercedes Rosas}
\address{Departamento de \'Algebra, Facultad de Matem\'aticas, Universidad de Sevilla, Avda.\ Reina Mercedes, Sevilla, Espa\~na}
\title[Kronecker coefficients]{On the growth of the Kronecker coefficients:
Accompanying Appendices.}
\date{\today}

\begin{document}

\maketitle
\appendix

This text is an appendix to our work "On the growth of Kronecker coefficients" \cite{LongVersion}.  Here, we provide some complementary theorems, remarks, and calculations that for the sake of space are not going to appear into the final version of our paper. 

We follow the same terminology and notation. External references to numbered equations, theorems, etc. are pointers to \cite{LongVersion}.  This file is not meant to be read independently of the main text.

\section{Bounds}\label{appbounds}

We prove here the assertions made in Remarks 5.3 
and 6.2 
about the values of the constants $k_i$ (in Theorem 5.2 
) and $k'_i$ (in Theorem 6.1
). These technical and less central results do not appear in the printed version of this work.

\subsection{Hook stability, reduced Kronecker coefficients}

In this section, we  find explicitly bounds for the quantities  $k_1$, $k_2$, $k_3$ appearing  in Theorem 5.2 

\begin{theorem} \label{bounds-columns}
In Theorem 5.2 
, one can take 
\[
\begin{array}{rcl}
k_1 &=& |\alpha|+\alpha_1+\beta'_1+\gamma'_1, \\
k_2 &=& |\beta|+\beta_1+\alpha'_1+\gamma'_1,   \\
k_3 &=& |\gamma|+\gamma_1+\alpha'_1+\beta'_1.
\end{array}
\]
\end{theorem}

\begin{proof}
From the proof of Theorem 5.2 
, one can take $k_i = \max_{\omega \in \Omega} \ell_i(\omega)$, where $\Omega$ is the support of $\polcol=\polcol_{\alpha,\beta,\gamma}(x,y,z)$.

Let us perform the change of variables $x=\frac{vw}{u}$, $y=\frac{uw}{v}$, $z=\frac{uv}{w}$, so that the identity $x^a y^b z^c = u^{\ell_1} v^{\ell_2} w^{\ell_3}$ holds. Then $k_1$ (resp. $k_2$, $k_3$), as defined above, is the degree of $P$ with respect to the variable $u$ (resp. $v$, $w$). 

After this change of variables, $H(-\varepsilon x,-\varepsilon y,-\varepsilon z)$ equals
\begin{multline*}
XYZ+XY+XZ+YZ
-\varepsilon \cdot\left( \frac{uv}{w} XY +\frac{uw}{v} XZ +\frac{vw}{u} YZ  \right)\\
+ u^2 X + v^2 Y + w^2 Z
+ \varepsilon \frac{1}{vw}\left( u^2-u v^2 - u w^2\right) X\\
+ \varepsilon \frac{1}{uw}\left( v^2-v u^2 - v w^2\right) Y
+ \varepsilon \frac{1}{uv}\left( w^2-w u^2 - w v^2\right) Z.
\end{multline*}
We reorder the terms as follows:
\[
H(-\varepsilon x,-\varepsilon y,-\varepsilon z)=u^2 X 
+ \varepsilon u X H_1 
-\varepsilon u \left(\frac{v}{w} Y+\frac{w}{v} Z \right) + H_0,
\]
where $H_0$ is a sum of monomials with non-positive degree in $u$, and $H_1$ is free of $u$ and $X$. We now factorize $\sigma[H(-\varepsilon x,-\varepsilon y,-\varepsilon z)]$ as
\[
\sigma[u^2 X] \cdot \sigma[\varepsilon u X H_1] \cdot \sigma\left[-\varepsilon u \frac{v}{w} Y \right]
\cdot \sigma\left[-\varepsilon u \frac{w}{v} Z \right] \cdot \sigma[H_0]
\]
 and expand each series, except $\sigma[H_0]$. We get that $\sigma[H(-\varepsilon x,-\varepsilon y,-\varepsilon z)]$ is equal to  
\[
\sum u^{2i} h_i[X]\, u^j e_j[X H_1]\, \left(\frac{v}{w}u\right)^k e_k[Y]\, \left(\frac{w}{v}u\right)^{\ell} e_\ell[Z] \sigma[H_0],
\]
where the sum ranges over all nonnegative integers $i$, $j$, $k$, $\ell$. Therefore,
\begin{multline*}
\polcol=\scalar{\sigma[H(-\varepsilon x,-\varepsilon y,-\varepsilon z)]}{s_{\alpha}[X] s_{\beta}[Y] s_{\gamma}[Z]}\\
= \sum u^{2i+j+k+\ell} v^{k-\ell} w^{\ell-k} \scalar{e_j[X H_1] \sigma[H_0]}{(h_i^{\perp} s_{\alpha})[X] (e_k^{\perp} s_{\beta})[Y] (e_{\ell}^{\perp} s_{\gamma}[Z])}.
\end{multline*}
We have $h_i^{\perp} s_{\alpha}=0$ unless $i \leq \alpha_1$, and that  $e_k^{\perp} s_{\beta}=0$ (resp. $e_\ell^{\perp} s_{\gamma}=0$) unless $j \leq \beta'_1$ (resp. $k \leq \gamma'_1$). Finally, since $e_j[X H_1]$ is homogeneous of degree $j$ in $X$ and $h_i^{\perp} s_{\alpha}$ has degree $|\alpha|-i$, the summand corresponding to $i$, $j$, $k$, $\ell$ can be non zero only if $i \leq \alpha_1$, $j \leq |\alpha|-i$, $k \leq \beta'_1$ and $\ell \leq \gamma'_1$. Therefore $2i+j+k+\ell \leq |\alpha|+\alpha_1+\beta'_1+\gamma'_1$.

This proves that in Theorem 5.2 
one can take $k_1=|\alpha|+\alpha_1+\beta'_1+\gamma'_1$. By symmetry, it follows that one can also take $k_2=|\beta|+\beta_1+\alpha'_1+\gamma'_1$ and $k_3=|\gamma|+\gamma_1+\alpha'_1+\beta'_1$.
\end{proof}

\begin{remark}
More detailed computations show that the coefficient of $u^{|\alpha|+\alpha_1+\beta'_1+\gamma'_1}$ in $\polcol_{\alpha,\beta,\gamma}$ is
\[
s_{{\overline{\alpha}}'}\left[ \frac{1+v^2+w^2}{vw}\right] \cdot s_{\beta'}[1] s_{\gamma'}[1] 
\]
where $\overline{\alpha}$ is the partition obtained from $\alpha$ by removing its first row (and ${\overline{\alpha}}'$ is the conjugate of $\overline{\alpha}$). This is non--zero if and only if $\beta$ and $\gamma$ have at most one column, and $\overline{\alpha}$ has at most three columns. This is the only case when the bound is reached. 
\end{remark}

\subsection{First row for reduced Kronecker coefficients.}  

 We give bounds for the constants $k'_1, k'_2$ and $k'_3$ appearing in Theorem 6.1. 

\begin{theorem} \label{effective bounds 2row}
In Theorem 6.1
, one can take 
\[
\begin{array}{rcl}
k'_1 &=& |\alpha|+|\beta|+|\gamma| + \beta_1 , \\
k'_2 &=&  |\alpha|+|\beta|+|\gamma| + \gamma_1 , \\
k'_3 &=&  |\alpha|+|\beta|+|\gamma| + \alpha_1+ \beta_1 + \gamma_1.
\end{array}
\]
\end{theorem}

\begin{proof}
From the proof of Theorem 6.1 
one can take 
\[
\begin{array}{rcl}
k'_1 &=& \max_{\omega \in \Omega} \ell_1(\omega),\\
k'_2 &=&  \max_{\omega \in \Omega} \left( \ell_1(\omega)-\ell_2(\omega)\right), \\
k'_3 &=& \max_{\omega \in \Omega} \left( \ell_1(\omega)-\ell_3(\omega)\right).
\end{array}
\]
Let us perform the change of variables $x = u v w$,  $ y = v w$, $z = u w$, so that $x^a y ^b z^c = u^{\ell_1-\ell_2} v^{\ell_1-\ell_3} w^{\ell_1}$. Then the constants $k'_1$, $k'_2$, $k'_3$ are the degrees of $\polrow_{\alpha, \beta, \gamma}$ in the variables, respectively, $u$, $v$ and $w$.

Let us bound the degree in $u$ of  $\polrow$. 
After the change of variables, we obtain that
$
H(x,y,z)
=
u^2 v w^2 Y + u H_1 + H_0
$
where $H_1$ is free of $u$ and has all its terms of degree at least $1$ in $X$,
$Y$ and $Z$, and $H_0$ has all its terms of degree $\leq 0$ in $u$. 
Thus,
\begin{align*}
\sigma[H]
= \sigma[u^2 v w^2 Y]  \sigma[u H_1]  \sigma[H_0]
= \sum_{i,j} u^{2 i + j} (v w^2)^i h_i[Y] h_j[H_1] \sigma[H_0],
\end{align*}
and, therefore,
\begin{align*}
\polrow{} &= \sum_{i,j} u^{2i+j} (v w^2)^j \scalar{ h_i[Y] h_j[H_1] \sigma[H_0]}{s_{\alpha}[X] s_{\beta}[Y] s_{\gamma}[Z]}\\
&=\sum_{i,j} u^{2i+j} (v w^2)^j \scalar{h_j[H_1] \sigma[H_0]}{s_{\alpha}[X]
(h_i^{\perp }s_{\beta})[Y] s_{\gamma}[Z]}.
\end{align*}
Note that $h_i^{\perp} s_{\beta} = 0$ unless $i \leq \beta_1$. Moreover the left--hand side of each scalar product in the sum is now a sum of homogeneous symmetric functions all of total degree at least  $j$, while the right--hand side has degree $|\alpha|+ |\beta|+ |\gamma| -i$. Thus,  the non-zero summands  fulfill $j \leq |\alpha|+|\beta+|\gamma| -i$.  We conclude that for all non--zero summands, $2i + j \leq |\alpha|+ |\beta|+ |\gamma|+ \beta_1$. 
\end{proof}


\section{Another approach to the hook stability property, derived from Murnaghan's stability and conjugation}\label{app:other}

Here we detail the arguments sketched in Section 5.3. 

\subsection{Hook Stability}

Recall that the Kronecker coefficients are invariant under conjugation of any two of their arguments:
\[
\K{\lambda}{\mu}{\nu}=\K{\lambda'}{\mu'}{\nu}=\K{\lambda'}{\mu}{\nu'}=\K{\lambda}{\mu'}{\nu'}.
\]
Assume that $(\lambda',\mu',\nu)$ is stable. That is, that the value of the Kronecker coefficient $\K{\lambda'}{\mu'}{\nu}$ does not change by adding one to the first parts of the three indexing partitions:
\begin{equation}\label{stab1}
\K{\lambda'}{\mu'}{\nu} = \K{\lambda'\op{1}{0}}{\mu'\op{1}{0}}{\nu\op{1}{0}}.
\end{equation}
Conjugating the partitions in position 1 and 2, we obtain that the Kronecker coefficient
$
\K{\lambda'\op{1}{0}}{\mu'\op{1}{0}}{\nu\op{1}{0}}$ is equal to $\K{\lambda\op{0}{1}}{\mu\op{0}{1}}{\nu\op{1}{0}}.$
We conclude that under our stability assumption
\[
\K{\lambda}{\mu}{\nu} = \K{\lambda\op{0}{1}}{\mu\op{0}{1}}{\nu\op{1}{0}}.
\]

Similarly, under the corresponding stability hypothesis for the triples $(\lambda', \mu, \nu')$ and  $(\lambda, \mu', \nu')$, we have
\[
\K{\lambda}{\mu}{\nu} = \K{\lambda\op{0}{1}}{\mu\op{1}{0}}{\nu\op{0}{1}} \text{ and }
\K{\lambda}{\mu}{\nu} = \K{\lambda\op{1}{0}}{\mu\op{0}{1}}{\nu\op{0}{1}}.
\]

Let $\lambda$, $\mu$ and $\nu$ be three non--empty partitions with the same weight $N$.
A sufficient condition for the stability of the triple $(\lambda, \mu, \nu)$ is
$N \geq N_0(\cut{\lambda}, \cut{\mu}, \cut{\nu}) $ (see (9), 
in Section 3.1). 
As a consequence of Lemma 5.6,
\begin{equation}\label{N0bis}
N \geq N_0(\cutt{\lambda}, \cutt{\mu}, \cutt{\nu}) + \frac{\lambda'_1+\mu'_1+\nu'_1}{2}
\end{equation}
is also a sufficient condition for a triple of non--empty partitions $(\lambda, \mu,\nu)$, all with weight $N$, to be stable. 
Let $\SetL$ be the set of triples of partitions of the same weight that fulfill \eqref{N0bis}. Let $\SetL_4$ be the set of triples of partitions $(\lambda, \mu, \nu)$ such that all four triples
\[
(\lambda, \mu, \nu), (\lambda', \mu', \nu), (\lambda', \mu, \nu'), (\lambda, \mu',\nu')
\]
are in $\SetL$. Then $\SetL_4$ is defined by  inequalities
\begin{equation}\label{SetL4}
\begin{matrix}
N &\geq& N_0(\cutt{\lambda}, \cutt{\mu}, \cutt{\nu}) +
(\lambda_1'+\mu_1'+\nu'_1)/2,\\
N &\geq& N_0(\cutt{\lambda}, \cutt{\mu}, \cutt{\nu}) + (\lambda_1+\mu_1+\nu'_1)/2, \\
N &\geq& N_0(\cutt{\lambda}, \cutt{\mu}, \cutt{\nu}) + (\lambda_1+\mu'_1+\nu_1)/2, \\
N &\geq& N_0(\cutt{\lambda}, \cutt{\mu}, \cutt{\nu}) + (\lambda_1'+\mu_1+\nu_1)/2,
\end{matrix}
\end{equation}
where, again, $N$ is the weight of the partitions $\lambda$, $\mu$ and $\nu$.

Fix partitions $\lambda$, $\mu$, $\nu$ of the same weight $N$.
Let $\Dom$ be the set of all $(a, b, c, m) \in \NN^4$ such that $m \geq a, b, c$
and  $(\lambda \op{m-a}{a}, \mu \op{m-b}{b}, \nu \op{m-c}{c})$ belongs to
$\SetL_4$. 
Then, from \eqref{SetL4}, it is straightforwardly calculated that  $\Dom$ is defined by the system of inequalities:
\[
\begin{matrix}
\ell_i(a,b,c) &\geq  \delta_i & \text{ for all } i\in\{1,2,3\},\\
m -(a+b+c)/2  &\geq  \delta,     & \\
m \geq a, b, c.
\end{matrix}
\]
with 
\[
\begin{matrix}
\delta_1  &=& 2 \; (N_0-N) + \lambda'_1+\mu_1+\nu_1, \\
\delta_2  &=& 2 \; (N_0-N) + \lambda_1+\mu'_1+\nu_1, \\
\delta_3  &=& 2 \; (N_0-N) + \lambda_1+\mu_1+\nu'_1, \\
\delta          &=& N_0 -N  + (\lambda'_1+\mu'_1+\nu'_1)/2,
\end{matrix}
\]
and $N_0 = N_0(\cutt{\lambda}, \cutt{\mu}, \cutt{\nu})$. 
Inequalities \eqref{SetL4} are just the inequalities (21) 
with $\delta_i$ and $\delta$ for $d_i$ and $d$. So, proving that the Kronecker coefficients $\K{\lambda \op{m-a}{a}}{\mu\op{m-b}{b}}{\nu \op{m-c}{c}}$ take only one value for $(a,b,c,m) \in \Dom$, would prove again Theorem 5.6. 
What we will get simply is that these Kronecker coefficients take at most two values.

Note that $\Dom$ is stable under addition of $u_1=(1,1,0,1)$, $u_2=(1,0,1,1)$, $u_3=(1,1,0,1)$ and $u_4=(0,0,0,1)$. The lattice spanned by these vectors is the set of all $(a,b,c,m) \in \ZZ^4$ such that $a+b+c \equiv 0 \pmod{2}$.
Consider $(a,b,c,m)$ and $(a',b',c',m')$ in $\Dom$, such that $a+b+c \equiv a'+b'+c' \pmod{2}$. Their difference is a linear combination with integer coefficients of the vectors $u_i$:
\[
(a,b,c,m) - (a', b', c', m') = \sum_{i=1}^4 x_i u_i.
\]
Set 
\[
(a'',b'',c'',m'')
= 
(a,b,c,m) + \sum_{i: x_i <0} (-x_i) u_i = (a', b', c', m') + \sum_{i: x_i > 0} x_i u_i.
\]
Then 
\[
\K{\lambda \op{m-a}{a}}{\mu\op{m-b}{b}}{\nu \op{m-c}{c}} 
=
\K{\lambda\op{m''-a''}{a''}}{\mu \op{m''-b''}{b''}}{\nu\op{m''-c''}{c''}}
\]
 and 
\[\K{\lambda \op{m'-a'}{a'}}{\mu\op{m'-b'}{b'}}{\nu \op{m'-c'}{c'}}
=
\K{\lambda\op{m''-a''}{a''}}{\mu \op{m''-b''}{b''}}{\nu\op{m''-c''}{c''}}.
\]
This shows that 
\[
\K{\lambda \op{m-a}{a}}{\mu\op{m-b}{b}}{\nu \op{m-c}{c}} = 
\K{\lambda \op{m'-a'}{a'}}{\mu\op{m'-b'}{b'}}{\nu \op{m'-c'}{c'}}.
\]
We conclude that the Kronecker coefficients $\K{\lambda \op{m-a}{a}}{\mu\op{m-b}{b}}{\nu \op{m-c}{c}}$ for $(a,b,c,m) \in \Dom$ take at most two values, one for each value of $a+b+c \pmod{2}$.

To recover fully Theorem 5.6, 
it would now be enough to show that for some $(a,b,c,m) \in \Dom$ such that $a+b+c \equiv 0 \pmod{2}$, we have that the Kronecker coefficients 
\[
\K{\lambda \op{m-a}{a}}{\mu\op{m-b}{b}}{\nu \op{m-c}{c}} 
\]
 and
\[
\K{\lambda \op{m-a+1}{a+1}}{\mu\op{m-b+1}{b+1}}{\nu \op{m-c+1}{c+1}}
\] 
are equal.
This would follow from Conjecture 5.10. 
Indeed, assume Conjecture 5.10 
holds. 
Let $M_0$ (resp. $M_1$) be the value taken by the Kronecker coefficients $\K{\lambda \op{m-a}{a}}{\mu\op{m-b}{b}}{\nu \op{m-c}{c}}$ for $(a,b,c,m) \in \Dom$ with $a+b+c$ even (resp. odd). 
Choose arbitrarily $(a,b,c,m) \in \Dom$ such that $a+b+c$ is even. Set $\alpha=\lambda \op{m-a}{a}$, $\beta =  \mu\op{m-b}{b}$ and $\gamma=\nu\op{n-c}{c}$. 
We have
\[
\K{\alpha}{\beta}{\gamma} 
\leq 
\K{\alpha \op{1}{1}}{\beta\op{1}{1}}{\gamma\op{1}{1}} 
\leq 
\K{\alpha\op{2}{2}}{\beta\op{2}{2}}{\gamma\op{2}{2}}.
\]
But this means $M_0 \leq M_1 \leq M_0$. Therefore $M_0=M_1$. 
 

\subsection{Monotonicity conjecture}

Let us recall the two conjectures of Section 5.3. 

 \begin{conjecture*}[Conjecture 5.10 
 restated]
For any three partitions $\lambda$, $\mu$ and $\nu$ of the same weight, 
\[
\K{\lambda}{\mu}{\nu} \leq \K{\lambda\op{1}{1}}{\mu \op{1}{1}}{\nu\op{1}{1}}.
\]
\end{conjecture*}

\begin{conjecture*}[Conjecture 5.11 
restated]
For any three partitions $\lambda$, $\mu$ and $\nu$ of the same weight, and any $(a,b,c,m)$ fulfilling (24), 
\[
\K{\lambda}{\mu}{\nu} \leq \K{\lambda\op{m-a}{a}}{\mu \op{m-b}{b}}{\nu\op{m-c}{c}}.
\]
\end{conjecture*}

Here we prove that  Conjecture 5.10 
implies the more general Conjecture 5.11. 

The proof of this implication is based, once again, on the invariance of the Kronecker coefficients under conjugating two arguments, and on Murnaghan's stability (see Section 3.1
): for any three partitions $\lambda$, $\mu$, $\nu$ of the same weight, 
\begin{equation}\label{ineq 111}
\K{\lambda}{\mu}{\nu} \leq \K{\lambda\op{1}{0}}{\mu \op{1}{0}}{\nu\op{1}{0}}.
\end{equation}

Assume that Conjecture 5.10 
holds. Let $\lambda$, $\mu$ and $\nu$ be three partitions of the same weight. 
We have the identity $\K{\lambda}{\mu}{\nu}= \K{\lambda'}{\mu'}{\nu}$. Using \eqref{ineq 111} we get
$
\K{\lambda'}{\mu'}{\nu} \leq \K{\lambda'\op{1}{0}}{\mu'\op{1}{0}}{\nu\op{1}{0}}.
$
Conjugating again the arguments in position 1 and 2, we have
that $
\K{\lambda'\op{1}{0}}{\mu'\op{1}{0}}{\nu\op{1}{0}}$ is equal to $\K{\lambda\op{0}{1}}{\mu\op{0}{1}}{\nu\op{1}{0}}.
$
Therefore, $
\K{\lambda}{\mu}{\nu} \leq \K{\lambda\op{0}{1}}{\mu\op{0}{1}}{\nu\op{1}{0}}.
$
Likewise
$
\K{\lambda}{\mu}{\nu} \leq \K{\lambda\op{0}{1}}{\mu\op{1}{0}}{\nu\op{0}{1}} $   and $\K{\lambda}{\mu}{\nu} \leq \K{\lambda\op{1}{0}}{\mu\op{0}{1}}{\nu\op{0}{1}}.$

Using these 3 identities, together with \eqref{ineq 111}, we see that $\K{\lambda}{\mu}{\nu} \leq \K{\lambda\op{m-a}{a}}{\mu\op{m-b}{b}}{\nu\op{n-c}{c}}$ for all $(a,b,c,m)$ in the semigroup $\semigroup$ generated by $(1,1,0,1)$, $(1,0,1,1)$, $(0,1,1,1)$ and $(0,0,0,1)$.
This semigroup $\semigroup$ is easily determined: it is the set of points $(a,b,c,m) \in \NN^4$ that fulfill (24) 
and $a+b+c \equiv 0 \pmod{2}$. The set of integer points fulfilling (24) 
splits in two classes: $\semigroup$ in the one hand, and $(1,1,1,2) + \semigroup$ in the other hand.

The implication follows now straightforwardly from this.


\section{The generating function for the coefficients
	$B_{\alpha,\beta,\gamma}$}\label{app:genfunc}

\subsection{Expression involving Schur functions indexed by hooks}

It is proved in Theorem 7.3 
that the Schur generating function for the coefficients  $B_{\alpha,\beta,\gamma}$ is
\[
\sigma[XYZ+2W] \cdot \left( \frac{3}{4}+\frac{1}{4}\sigma[(\varepsilon-1) W]-\frac{1}{2} \chi[W] + \chi[YZ-X] \right)
\]
The following result is stated in Remark 7.4, 
with no proof.
\begin{proposition}
Fix partitions $\alpha$, $\beta$, $\gamma$. 
The coefficient $B_{\alpha,\beta,\gamma}$ in Theorem 6.1 
is the coefficient of $s_{\alpha}[X] s_{\beta}[Y] s_{\gamma}[Z]$ in the expansion in the Schur basis of
\[
\sigma[XYZ+2W]  \cdot \left(1 - \sum_{a \text{even }, b} (-1)^b s_{(a|b)}[W] + \sum_{a, b} (-1)^b s_{(a|b)}[YZ-X] \right).
\]
\end{proposition}

\begin{proof}
From Cauchy's Formula,
\begin{multline}
\sigma[(\varepsilon-1) W] = 
\sigma[(1-\varepsilon) (-W)] = \\
\sum_{\lambda} s_{\lambda}[1-\varepsilon] s_{\lambda}[-W]=
\sum_{\lambda} s_{\lambda}[1-\varepsilon] (-1)^{|\lambda|}s_{\lambda'}[W].
\end{multline}
From \cite[Ex. 7.43 with $t=1$]{Stanley}, $s_{\lambda}[1-\varepsilon]$ is $1$ if $\lambda$ is the empty partition, $2$ if $\lambda$ is a hook and $0$ otherwise. Therefore,
\[
\sigma[(\varepsilon-1) W] = 1+2\;\sum_{a, b \geq 0} (-1)^{1+a+b} s_{(a|b)}[W].
\]
Thus,
\[
\frac{3}{4}+\frac{1}{4}\sigma[(\varepsilon-1) W] = 1 + \frac{1}{2} \sum_{a, b} (-1)^{1+a+b} s_{(a|b)}[W].
\]
From \cite[I.\S 3. Ex. 11 (2) with $\mu=\ep$]{Macdonald}, we have 
\begin{equation}\label{chi}
\chi=\sum_{a, b} (-1)^{b} s_{(a|b)},
\end{equation}
where $\chi$ is the sum of the power sum symmetric functions as defined in
Proposition 7.3. 
Therefore,
\begin{multline*}
\frac{3}{4}+\frac{1}{4}\sigma[(\varepsilon-1) W] -\frac{1}{2} \chi[W] 
=\\
1+ \frac{1}{2} \sum_{a, b} (-1)^{1+a+b} s_{(a|b)}[W] -\frac{1}{2}\sum_{a, b} (-1)^{b} s_{(a|b)}\\
=
1 -  \sum_{a \text{ even}, b} (-1)^b s_{(a|b)}[W].
\end{multline*}
Using again \eqref{chi} to rewrite $\chi[YZ-X]$, we get the following formula for the generating function of the coefficients of $B$:
\[
\sigma[XYZ+2W]  \cdot \left(1 - \sum_{a \text{even }, b} (-1)^b s_{(a|b)}[W] +
\sum_{a, b} (-1)^b s_{(a|b)}[YZ-X] \right).
 \]
\end{proof}

\subsection{Toolbox for other expressions}

In order to write in other ways the generating function for the coefficients $B_{\alpha,\beta, \gamma}$, the following formulas may be useful:
\begin{align*}
\sigma[X] \cdot \chi[X] &= \sum_{k} k \, h_k[X],\\  
\sigma[2X] \cdot \chi[X] &= \sum_{\lambda: \ell(\lambda) \leq 2} \frac{(\lambda_1-\lambda_2+1)(\lambda_1+\lambda_2)}{2} \, s_{\lambda}[X].
\end{align*}
They follow from the fact that $\sigma[t X] \chi[tX]$ is the derivative of
$\sigma[tX]$ (for the first one), and that $\sigma[2 t X] \chi[ 2 tX]$ is the
derivative of $\sigma[2 tX]$. Last, by Cauchy's formula,
\[
\sigma[2 t X]= \sum_{\lambda} s_{\lambda}[2] s_{\lambda}[X] t^{|\lambda|}, 
\]
and $s_{\lambda}[2] = (\lambda_1-\lambda_2+1)$ if $\lambda$ has at most two parts, and is equal to $0$ otherwise.

\section{Table of coefficients}\label{apptab}

Tables \ref{part1} and \ref{part2} display the coefficients $\SSK{\alpha}{\beta}{\gamma}$, $A_{\alpha,\beta,\gamma}$, $ B_{\alpha,\beta,\gamma}$ and $C_{\alpha,\beta,\gamma}$ for all partitions $\alpha$, $\beta$ and $\gamma$ with weight at most $3$. Note that $\SSK{\alpha}{\beta}{\gamma}$, $A_{\alpha,\beta,\gamma}$ and $C_{\alpha,\beta,\gamma}$ are invariant under permutation of their three indices. This is why the table gives their values only for $\alpha \geq \beta \geq \gamma$, where the order $\geq$ is the degree lexicographic ordering. The coefficients $B_{\alpha,\beta,\gamma}$ is only invariant under permuting its last two indices. 

These coefficients where calculated by series expansion of the generating series and using SAGE \cite{SAGE}.

\begin{table}
\[
\begin{array}{ccccccccc}
\alpha & \beta & \gamma & \SSK{\alpha}{\beta}{\gamma} & A_{\alpha,\beta,\gamma} & B_{\alpha,\beta,\gamma} & B_{\beta,\alpha,\gamma} & B_{\gamma,\alpha,\beta} & C_{\alpha,\beta,\gamma} \\\hline
\ep       &\ep       &\ep       &1    &1    &1   
&1    &1    &1     \\
(1)      &\ep       &\ep       &2    &2    &   
0&1    &1    &    0 \\
(1)      &(1)      &\ep       &6    &6    &    0&
0&3    &    0 \\
(1)      &(1)      &(1)      &21   &21   &    0&
0&    0&1     \\
(2)      &\ep       &\ep       &2    &3    &-2  
&1    &1    &1     \\
(2)      &(1)      &\ep       &8    &10   &-7  
&-2   &3    &    0 \\
(2)      &(1)      &(1)      &34   &40   &-25 
&-5   &-5   &    0 \\
(2)      &(2)      &\ep       &14   &20   &-14 
&-14  &6    &2     \\
(2)      &(2)      &(1)      &66   &86   &-57 
&-57  &-14  &    0 \\
(2)      &(2)      &(2)      &145  &203  &-133
&-133 &-133 &5     \\
(2)      &(2)      &(1, 1)   &144  &150  &-84 
&-84  &-84  &-4    \\
(2)      &(2)      &(1, 1, 1)&204  &134  &-54 
&-54  &-121 &    0 \\
(2)      &(1, 1)   &\ep       &14   &12   &-8  
&-8   &4    &-2    \\
(2)      &(1, 1)   &(1)      &66   &62   &-33 
&-33  &-2   &    0 \\
(2)      &(1, 1)   &(1, 1)   &145  &131  &-55 
&-55  &-55  &5     \\
(2)      &(1, 1, 1)&\ep       &16   &6    &-3  
&-6   &3    &    0 \\
(2)      &(1, 1, 1)&(1)      &84   &46   &-19 
&-42  &4    &    0 \\
(2)      &(1, 1, 1)&(1, 1)   &206  &144  &-45 
&-117 &-45  &    0 \\
(2)      &(1, 1, 1)&(1, 1, 1)&326  &240  &-48 
&-168 &-168 &    0 \\
(1, 1)   &\ep       &\ep       &2    &1    &-1   &
0&    0&-1    \\
(1, 1)   &(1)      &\ep       &8    &6    &-3   &
0&3    &    0 \\
(1, 1)   &(1)      &(1)      &34   &28   &-13 
&1    &1    &    0 \\
(1, 1)   &(1, 1)   &\ep       &14   &12   &-4  
&-4   &8    &2     \\
(1, 1)   &(1, 1)   &(1)      &66   &54   &-21 
&-21  &6    &    0 \\
(1, 1)   &(1, 1)   &(1, 1)   &144  &110  &-38 
&-38  &-38  &-4    \\
(3)      &\ep       &\ep       &2    &4    &-6   &
0&    0&    0 \\
(3)      &(1)      &\ep       &8    &14   &-20 
&-6   &1    &    0 \\
(3)      &(1)      &(1)      &38   &59   &-78 
&-19  &-19  &1     \\
(3)      &(2)      &\ep       &16   &30   &-42 
&-27  &3    &    0 \\
(3)      &(2)      &(1)      &84   &138  &-178
&-109 &-40  &    0 \\
(3)      &(2)      &(2)      &206  &348  &-435
&-261 &-261 &    0 \\
(3)      &(2)      &(1, 1)   &204  &258  &-299
&-170 &-170 &    0 \\
(3)      &(2)      &(1, 1, 1)&320  &250  &-250
&-125 &-250 &    0 \\
(3)      &(1, 1)   &\ep       &16   &18   &-24 
&-15  &3    &    0 \\
(3)      &(1, 1)   &(1)      &84   &98   &-118
&-69  &-20  &    0 \\
(3)      &(1, 1)   &(1, 1)   &206  &220  &-235
&-125 &-125 &    0 \\
(3)      &(3)      &\ep       &22   &50   &-72 
&-72  &3    &    0 \\
(3)      &(3)      &(1)      &122  &240  &-321
&-321 &-81  &2     \\
(3)      &(3)      &(2)      &326  &640  &-820
&-820 &-500 &    0 \\
(3)      &(3)      &(1, 1)   &320  &478  &-574
&-574 &-335 &    0 \\
(3)      &(3)      &(3)      &565  &1243
&-1597&-1597&-1597&5     \\
(3)      &(3)      &(2, 1)   &1056 &1632
&-1888&-1888&-1888&    0 \\
(3)      &(3)      &(1, 1, 1)&544  &506  &-521
&-521 &-521 &-4   
\end{array}
\]
\caption{Table of the coefficients of the paper, for three indexing partitions with weight at most $3$ (part 1 of 2). }\label{part1}
\end{table}

\begin{table}
\[
\begin{array}{ccccccccc}
\alpha & \beta & \gamma & \SSK{\alpha}{\beta}{\gamma} & A_{\alpha,\beta,\gamma} & B_{\alpha,\beta,\gamma} & B_{\beta,\alpha,\gamma} & B_{\gamma,\alpha,\beta} & C_{\alpha,\beta,\gamma} \\\hline
(3)      &(2, 1)   &\ep       &38   &50   &-66 
&-66  &9    &    0 \\
(3)      &(2, 1)   &(1)      &224  &288  &-344
&-344 &-56  &    0 \\
(3)      &(2, 1)   &(2)      &610  &824  &-938
&-938 &-526 &    0 \\
(3)      &(2, 1)   &(1, 1)   &610  &700  &-738
&-738 &-388 &    0 \\
(3)      &(2, 1)   &(2, 1)   &2037 &2465
&-2515&-2515&-2515&1     \\
(3)      &(2, 1)   &(1, 1, 1)&1056 &928  &-832
&-832 &-832 &    0 \\
(3)      &(1, 1, 1)&\ep       &22   &10   &-12 
&-12  &3    &    0 \\
(3)      &(1, 1, 1)&(1)      &122  &80   &-85 
&-85  &-5   &-2    \\
(3)      &(1, 1, 1)&(1, 1)   &326  &260  &-240
&-240 &-110 &    0 \\
(3)      &(1, 1, 1)&(1, 1, 1)&565  &451  &-355
&-355 &-355 &5     \\
(2, 1)   &\ep       &\ep       &2    &2    &-3   &
0&    0&    0 \\
(2, 1)   &(1)      &\ep       &12   &12   &-15 
&-3   &3    &    0 \\
(2, 1)   &(1)      &(1)      &64   &64   &-72 
&-8   &-8   &    0 \\
(2, 1)   &(2)      &\ep       &28   &30   &-36 
&-21  &9    &    0 \\
(2, 1)   &(2)      &(1)      &152  &164  &-181
&-99  &-17  &    0 \\
(2, 1)   &(2)      &(2)      &382  &442  &-477
&-256 &-256 &    0 \\
(2, 1)   &(2)      &(1, 1)   &382  &378  &-371
&-182 &-182 &    0 \\
(2, 1)   &(2)      &(1, 1, 1)&610  &472  &-394
&-158 &-394 &    0 \\
(2, 1)   &(1, 1)   &\ep       &28   &26   &-28 
&-15  &11   &    0 \\
(2, 1)   &(1, 1)   &(1)      &152  &140  &-139
&-69  &1    &    0 \\
(2, 1)   &(1, 1)   &(1, 1)   &382  &330  &-293
&-128 &-128 &    0 \\
(2, 1)   &(2, 1)   &\ep       &74   &74   &-81 
&-81  &30   &    0 \\
(2, 1)   &(2, 1)   &(1)      &428  &428  &-433
&-433 &-5   &    0 \\
(2, 1)   &(2, 1)   &(2)      &1168 &1242
&-1218&-1218&-597 &    0 \\
(2, 1)   &(2, 1)   &(1, 1)   &1168 &1094 &-982
&-982 &-435 &    0 \\
(2, 1)   &(2, 1)   &(2, 1)   &3933 &3933
&-3470&-3470&-3470&1     \\
(2, 1)   &(2, 1)   &(1, 1, 1)&2037 &1609
&-1221&-1221&-1221&1     \\
(2, 1)   &(1, 1, 1)&\ep       &38   &26   &-24 
&-24  &15   &    0 \\
(2, 1)   &(1, 1, 1)&(1)      &224  &160  &-136
&-136 &24   &    0 \\
(2, 1)   &(1, 1, 1)&(1, 1)   &610  &444  &-338
&-338 &-116 &    0 \\
(2, 1)   &(1, 1, 1)&(1, 1, 1)&1056 &736  &-480
&-480 &-480 &    0 \\
(1, 1, 1)&\ep       &\ep       &2    &    0&    0&
0&    0&    0 \\
(1, 1, 1)&(1)      &\ep       &8    &2    &-2   &
0&1    &    0 \\
(1, 1, 1)&(1)      &(1)      &38   &17   &-15 
&2    &2    &-1    \\
(1, 1, 1)&(1, 1)   &\ep       &16   &10   &-8  
&-3   &7    &    0 \\
(1, 1, 1)&(1, 1)   &(1)      &84   &54   &-42 
&-15  &12   &    0 \\
(1, 1, 1)&(1, 1)   &(1, 1)   &204  &134  &-97 
&-30  &-30  &    0 \\
(1, 1, 1)&(1, 1, 1)&\ep       &22   &18   &-12 
&-12  &15   &    0 \\
(1, 1, 1)&(1, 1, 1)&(1)      &122  &88   &-59 
&-59  &29   &2     \\
(1, 1, 1)&(1, 1, 1)&(1, 1)   &320  &206  &-130
&-130 &-27  &    0 \\
(1, 1, 1)&(1, 1, 1)&(1, 1, 1)&544  &322  &-175
&-175 &-175 &-4
\end{array}
\]
\caption{Table of the coefficients of the paper, for three indexing partitions with weight at most $3$ (part 2 of 2).}
\label{part2}
\end{table}


\bibliographystyle{hplain}
\bibliography{kroproduct}

\end{document}